\documentclass{amsart}
\usepackage[T1]{fontenc}
\usepackage{ucs}
\usepackage[utf8x]{inputenc}
\usepackage{color}
\usepackage[enableskew,vcentermath]{youngtab}
\usepackage{amsfonts}
\usepackage{hhline}
\usepackage{amsthm}
\usepackage{amsmath}
\usepackage{amssymb}

\begin{document}

\title[Base And Cover Partitions]{On Base Partitions And Cover Partitions Of Skew Characters}
\author[C. Gutschwager]{Christian Gutschwager}
\address{Institut für Algebra, Zahlentheorie und Diskrete Mathematik, Leibniz Universität Hannover,  Welfengarten 1, D-30167 Hannover}
\email{gutschwager (at) math (dot) uni-hannover (dot) de}

\newtheorem{Le}{Lemma}[section]
\newtheorem{Ko}[Le]{Lemma}
\newtheorem{Sa}[Le]{Theorem}
\newtheorem{pro}[Le]{Proposition}
\newtheorem{Bem}[Le]{Remark}
\newtheorem{Def}[Le]{Definition}
\newtheorem{Bsp}[Le]{Example}
\renewcommand{\l}{\lambda}
\newcommand{\bl}{\bar\lambda}
\newcommand{\bn}{\bar\nu}
\newcommand{\mA}{\mathcal{A}}
\newcommand{\mB}{\mathcal{B}}
\newcommand{\mC}{\mathcal{C}}
\newcommand{\mD}{\mathcal{D}}
\renewcommand{\a}{\alpha}
\renewcommand{\b}{\beta}
\renewcommand{\c}{\gamma}
\newcommand{\C}{\Gamma}
\renewcommand{\k}{\kappa}
\newcommand{\U}{\Upsilon}
\newcommand{\h}{\hfil}
\newcommand{\X}{X}
\renewcommand{\pm}[1]{\begin{pmatrix}#1\end{pmatrix}}
\newcommand{\pinw}{{\pi_{nw}}}
\newcommand{\abs}[1]{\lvert #1 \rvert}
\newcommand{\tm}{\tilde\mu}
\newcommand{\tn}{\tilde\nu}
\newcommand{\lm}{\l/\mu}
\newcommand{\m}{\mu}
\newcommand{\n}{\nu}
\newcommand{\ab}{\a/\b}


\begin{abstract}
In this paper we give an easy combinatorial description for the base partition $\mB$ of a skew character $[\mA]$, which is the intersection of all partitions $\a$ whose corresponding character $[\a]$ appears in $[\mA]$.

This we use to construct the cover partition $\mC$ for the ordinary outer product as well as for the Schubert product of two characters and for some skew characters, here the cover partition is the union of all partitions whose corresponding character appears in the product or in the skew character.

This gives us also the Durfee size for arbitrary Schubert products.
\end{abstract}
\maketitle

\section{Introduction}
In this paper we give upper and lower bounds for partitions $\a$ such that the corresponding character $[\a]$ can appear in a given outer product or Schubert product of two characters or in a skew character.

We give in Remark~\ref{Bem:kurz} an easy combinatorial description for the base partition $\mB$ of a skew character $[\mA]$, which is the intersection of all partitions $\a$ whose corresponding character $[\a]$ appears in $[\mA]$.

Using~\cite[Theorem 4.2]{Gut} and the base partition we construct the cover partition $\mC$ for the ordinary outer product as well as for the Schubert product of two characters and for some skew characters. Here the cover partition is the union of all partitions whose corresponding character appears in the product or in the skew character.

\section{Notation and Littlewood Richardson Symmetries}
We mostly follow the standard notation in \cite{Sag}. A partition $\l=(\l_1,\l_2,\ldots,\l_l)$ is a weakly decreasing sequence of non-negative integers, $\l_i$ called the parts of $\l$. For the length we write $l(\l)=l$ and the sum $\left|\l\right|=\sum_i \l_i$. With a partition $\l$ we associate a diagram, which we also denote by $\l$, containing $\l_i$ left-justified boxes in the $i$-th row and we use matrix style coordinates to refer to the boxes. Sometimes we will use the short notation $\l=(\l_1^{l_1},\l_2^{l_2},\ldots)$ which means $\l$ has $l_1$ times the part $\l_1$, $l_2$ times the part $\l_2$ and so forth.

The conjugate $\l^c$ of $\l$ is the diagram which has $\l_i$ boxes in the $i$-th column.

For $\mu \subseteq \l$ we define the skew diagram $\lm$ as the difference of the diagrams $\l$ and $\mu$ defined as the difference of the set of the boxes. Rotation of $\lm$ by $180^\circ$ yields a skew diagram $(\lm)^\circ$ which is well defined up to translation. A skew tableau $T$ is a skew diagram in which the boxes are replaced by positive integers.  We refer with $T(i,j)$ to the entry in box $(i,j)$. A semistandard tableau of shape $\lm$ is a filling of $\lm$ with positive integers such that the following inequalities hold for all $(i,j)$ for which they are defined: $T(i,j)<T(i+1,j)$ and $T(i,j)\leq T(i,j+1)$. The content of a semistandard tableau $T$ is $\nu=(\nu_1,\ldots)$ if the number of occurrences of the entry $i$ in $T$ is $\nu_i$. The reverse row word of a tableau $T$ is the sequence obtained by reading the entries of $T$ from right to left and top to bottom starting at the first row. Such a sequence is said to be a lattice word if for all $i,n \geq1$ the number of occurrences of $i$ among the first $n$ terms is at least the number of occurrences of $i+1$ among these terms. The Littlewood Richardson (LR-) coefficient $c(\lambda;\mu,\nu)$ equals the number of semistandard tableaux of shape $\lm$ with content $\nu$ such that the reverse row word is a lattice word. We will call those tableaux LR-tableaux. The LR-coefficients play an important role in different contexts (see \cite{Sag}).

The irreducible characters $[\l]$ of the symmetric group $S_n$ are naturally labeled by partitions $\l\vdash n$. The skew character $[\lm]$ to a skew diagram $\lm$ is defined by the LR-coefficients:
\[ [\lm]=\sum_\nu c(\lambda;\mu,\nu) [\nu] \]

If $c(\lambda;\mu,\nu)\neq0$ we say that $[\nu]$ appears in $[\lm]$ and write $[\nu]\in[\lm]$. If $c(\lambda;\mu,\nu)=0$ we write instead $[\nu]\notin[\lm]$.

There are many known symmetries of the LR-coefficients (see \cite{Sag}).

We have that $c(\lambda;\mu,\nu)=c(\lambda;\nu,\mu)$. The translation symmetry gives $[\lm]=[\ab]$ if the skew diagrams of $\lm$ and $\ab$ are the same up to translation while rotation symmetry gives $[(\lm)^\circ]=[\lm]$. Furthermore the conjugation symmetry $c(\l^c;\m^c,\n^c)=c(\l;\m,\n)$ is also well known.

Furthermore rearranging the parts of skew diagram $\lm$ gives a partition $\a$ with $[\a]$ appearing in $[\lm]$. This follows for example easily by writing into the conjugated skew diagram $(\lm)^c$ into each column the entries $1$ to the length of this column. This means $[\a^c]$ appears in $[(\lm)^c]$ and conjugating again gives that $[\a]$ appears in $[\lm]$. We will use this fact later on.

We say that a skew diagram $\mD$ decays into the disconnected skew diagrams $\mA$ and $\mB$ if no box of $\mA$ (viewed as boxes in $\mD$) is in the same row or column as a box of $\mB$. We write $\mD=\mA\otimes\mB$ if $\mD$ decays into $\mA$ and $\mB$. A skew diagram is connected if it does not decay.

A skew character whose skew diagram $\mD$ decays into disconnected (skew) diagrams $\mA,\mB$ is equivalent to the product of the  characters of the disconnected diagrams induced to a larger symmetric group. We have  \[[\mD]=([\mA]\times[\mB])\uparrow_{S_m\times S_n}^{S_{n+m}}=:[\mA]\otimes[\mB]\] with $\abs{\mA}=m,\abs{\mB}=n$.  If $\mD=\lm$ and $\mA,\mB$ are proper partitions $\a,\b$ we have:
\[[\lm]= \sum_\nu c(\l;\mu,\nu)[\nu]=\sum_\nu c(\nu;\a,\b)[\nu] =[\a]\otimes[\b]\]

In the cohomology ring $H^*(Gr(l,\mathbb{C}^n),\mathbb{Z})$ of the Grassmannian $Gr(l,\mathbb{C}^n)$ of $l$-di\-men\-sio\-nal subspaces of $\mathbb{C}^n$ the product of two Schubert classes $\sigma_\a, \sigma_\b$ is given by:

\[\sigma_\a\cdot\sigma_\b=\sum_{\nu\subseteq((n-l)^l)}c(\nu;\a,\b)\sigma_\nu\]

In~\cite[Section4]{Gut} we established a close connection between the Schubert product and skew characters. To use this relation later on we define the Schubert product for characters in the obvious way as a restriction of the ordinary product:

\[[\a]\star_{(k^l)}[\b]:=\sum_{\nu\subseteq(k^l)} c(\nu;\a,\b)[\nu] \]

The Durfee size $d(\l)$ of a partition $\l$ is $d$ if $(d^d)\subseteq\l$ is the largest square contained in $\l$. The Durfee size of a character $\chi$ is the biggest Durfee size of all partitions whose corresponding character appears in the decomposition of $\chi$:
\[d(\chi)=\max(d(\nu)\,|\,[\nu]\in\chi)\]

\section{The Base Partition}
In this section we give in Remark~\ref{Bem:kurz} an easy combinatorial description for the base partition $\mB$ of a skew character $[\mA]$, which is the intersection of all partitions $\a$ whose corresponding character $[\a]$ appears in $[\mA]$. Since we have that skew characters $[\mA]$ whose skew diagram decays into two partitions $\mA=\mu\otimes\nu$ satisfies $[\mA]=[\mu]\otimes[\nu]$ this gives us also a combinatorial description for $\mB([\mu]\otimes[\nu])$

\begin{Def}
 We say that a partition $\a$ is contained in $\mA$ if there is a subdiagram of $\mA$ which is $\a$ or $\a^\circ$.
\end{Def}

For example the skew diagram $\mA=(11,6,5^3,4)/(3^2)$
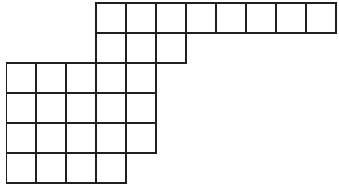
\[\young(:::\h\h\h\h\h\h\h\h,:::\h\h\h,\h\h\h\h\h,\h\h\h\h\h,\h\h\h\h\h,\h\h\h\h)\]

contains the partitions $\a^1=(8,3,2^3,1), \a^2=(5^3,4), \a^3=(5^3,2^2), \a^4=(4^4,1^2)$:
{ \[\a^1:\young(:::XXXXXXXX,:::XXX,\h\h\h\X\X,\h\h\h\X\X,\h\h\h\X\X,\h\h\h\X)\qquad         \a^2:\young(:::\h\h\h\h\h\h\h\h,:::\h\h\h,XXXXX,XXXXX,XXXXX,XXXX)\] \[\a^3:\young(:::\X\X\h\h\h\h\h\h,:::\X\X\h,XXXXX,XXXXX,XXXXX,\h\h\h\h) \qquad
  \a^4:\young(:::X\h\h\h\h\h\h\h,:::X\h\h,XXXX\h,XXXX\h,XXXX\h,XXXX)\]}

All other partitions contained in $\mA$ are subdiagrams of some $\a^i$.

\begin{Def}
For a skew diagram $\mA$ define the union partition $\U(\mA)$ as the union of all partitions $\a$ which are contained in $\mA$. So:
\[\U(\mA)_i=\max(\a_i\;\lvert \;\a \textnormal{ is contained in } \mA )\]
Since an arbitrary partition $\a$ is the union of the rectangles $((\a_i)^i)$ (for example we have that $\a=(5,3,3,2,1)$ is the union of the five rectangles $(5^1),(3^2),(3^3),(2^4)$ and $(1^5)$), it is sufficient for $\U$ to restrict the union to all rectangles $\a$ which are contained in $\mA$ but are not contained in a larger rectangle $\b$ contained in $\mA$.
\end{Def}

So in the above example we have $\U(\mA)=(8,5,5,4,2,1)$ and there are $6$ of those maximal rectangles.

\begin{Def}
 For a character $\chi$ we define the base partition $\mB(\chi)$ as the intersection of all $\a$ with $[\a]\in\chi$. So:
\[ \mB(\chi)_i=\min(\a_i \;\lvert \; [\a]\in\chi) \]
\end{Def}

\begin{Sa}\label{Sa:main}
 Let $\mA$ be a skew diagram.

 Then: $\U(\mA)=\mB([\mA])$
\end{Sa}
\begin{proof}
From the LR-rule follows that if a rectangle $(m^l)$ is contained in $\mA$ then $(m^l)$ is contained in every partition $\nu$ with $[\nu]\in[\mA]$. So we have $\U\subseteq \mB=\mB([\mA])$.

We will show by induction on the biggest length of a column contained in $\mA$ that the lower bound for the rows of $\mB$ is reached for some partitions $\a$ with $[\a]\in[\mA]$.

Let us assume that the biggest length of a column in $\mA$ is $1$. Then $\mA$ decomposes into disconnected rows and $\U$ is the biggest row contained in $\mA$.

We have a character $[\a]\in[\mA]$ such that $\a$ contains the parts of $[\mA]$ and so $\a_1=\U_1$ which gives $\U_1=\mB_1$.

On the other hand if we place $1$s into every box of $\mA$ we obtain a LR-tableau and so we have $[n]\in[\mA]$ with $n=\abs{\mA}$ which gives us $\mB_2=0$ and so $\U=\mB$.

Let us now assume that $\U=\mB$ holds for all skew diagrams which have columns of length not larger than $n-1$ and let $\mA$ be a skew diagram which has one or more columns of length $n$.

Rearranging the parts of $\mA$ gives again a partition $\a$ whose corresponding character satisfies $[\a]\in[\mA]$ and $\a_1=\U_1$ and so again $\U_1=\mB_1$.

We will now prove that $\U_i=\mB_i$ holds also for $i\geq2$.

We have a $1-1$-relation between the LR-tableaux $A$ of shape $\mA$ and $1$s in the top boxes of every column and arbitrary LR-tableaux $D$ of shape $\mD$, where $\mD$ is the skew diagram $\mA$ with the top boxes of every column removed, simply by removing all $1$s from $A$ and replacing the entry $i$ in $A$ with $i-1$.

If we have for example $\mA=(6^3,5,4^2)/(4^2,1^2)$ and some arbitrary LR-filling we get with the above construction:

\[\mA=\young(::::\h\h,::::\h\h,:\h\h\h\h\h,:\h\h\h\h,\h\h\h\h,\h\h\h\h) \qquad A=\young(::::11,::::22,:11133,:2224,1335,2446) \longleftrightarrow D=\young(::::11,::::22,:1113,:224,1335)\]

Removing the top boxes of each column from $\mA$ reduces each of the maximal rectangles $\a^i$ by one row and gives $\hat\a^i$ which is then one of the maximal rectangles in $\mD$. So we get $\U(\mA)_{i+1}=\U(\mD)_{i}$.

In the above example we have the following rectangles $\a^i$ in $\mA$ and $\hat\a^i$ in $\mD$:

\[\a^1:\young(::::XX,::::XX,:\h\h\h\X\X,:\h\h\h\h,\h\h\h\h,\h\h\h\h)\quad
\a^2:\young(::::\X\h,::::\X\h,:\h\h\h\X\h,:\h\h\h\X,\h\h\h\h,\h\h\h\h)\quad \a^3:\young(::::\h\h,::::\h\h,:XXXXX,:\h\h\h\h,\h\h\h\h,\h\h\h\h)\]
 \[\hat\a^1:\young(::::XX,::::\X\X,:\h\h\h\h,:\h\h\h\h,\h\h\h\h)\quad
\hat\a^2:\young(::::\X\h,::::\X\h,:\h\h\h\X,:\h\h\h,\h\h\h\h)\quad
\hat\a^3:\young(::::\h\h,::::\h\h,:\h\h\h\h,:\h\h\h,\h\h\h\h)\]
\vspace{0.5cm}
\[\a^4:\young(::::\h\h,::::\h\h,:\X\X\X\X\h,:\X\X\X\X,\h\h\h\h,\h\h\h\h)\quad
\a^5:\young(::::\h\h,::::\h\h,:\X\X\X\h\h,:\X\X\X\h,\h\X\X\X,\h\X\X\X)\quad
\a^6:\young(::::\h\h,::::\h\h,:\h\h\h\h\h,:\h\h\h\h,\X\X\X\X,\X\X\X\X)\]

\[\hat\a^4:\young(::::\h\h,::::\h\h,:\X\X\X\X,:\h\h\h,\h\h\h\h)\quad
\hat\a^5:\young(::::\h\h,::::\h\h,:\X\X\X\h,:\X\X\X,\h\X\X\X)\quad
\hat\a^6:\young(::::\h\h,::::\h\h,:\h\h\h\h,:\h\h\h,\X\X\X\X)\]

Since the biggest length of columns in $\mD$ is $n-1$ we have $\U(\mD)=\mB([\mD])$. For $i\geq1$ let $[\hat\b^i]\in[\mD]$  with $\hat\b^i_i=\mB([\mD])_i$. Since the characters $[\hat\b^i]\in[\mD]$ correspond to characters $[\b^i]\in[\mA]$, with $\b^i=(j,\hat \b^i)$ and $j$ equal to the number of columns in $\mA$, we have in $[\mA]$ characters $[\b^i]$ with:
\[\b^i_{i+1}=\hat\b^i_i=\U(\mD)_i=\U(\mA)_{i+1}\]

This gives $\U(\mA)=\mB([\mA])$. \end{proof}

The previous proof also gives us the following description for the base partition $\mB([\mA])$.

\begin{Bem}\label{Bem:kurz}
Let $\rho^i(\mA)$ be the skew diagram obtained from $\mA$ by removing in every column the top $i-1$ boxes.

Then we have that the $i$-th part $\mB_i$ of the base partition $\mB$ is the maximal part of $\rho^i(\mA)$, and so for $\mA=\lm$ we have $\mB_i=\max_{j} \abs{\l_{i+j-1}-\m_j}_+$ with $\abs{x}_+=\max(0,x)$.

Also $\rho^i(\mA)$ is the $i$-row overlap composition defined in \cite{RSW}. There it was also proved that equality of $[\mA^1]=[\mA^2]$ for skew diagrams $\mA^1,\mA^2$ requires that $\rho^i(\mA^1)$ and $\rho^i(\mA^2)$ must have the same parts in the same quantity for every $i$. This follows easily from the $1-1$ correspondence used also in the proof of Theorem~\ref{Sa:main}. In \cite{MN} these $\rho^i(\mA)$ are used to get necessary conditions for positivity of $[\mA^1]-[\mA^2]$.
\end{Bem}

\section{The Cover Partition}
In this section we use Remark~\ref{Bem:kurz} and \cite[Theorem 4.2]{Gut} to determine for the ordinary and the Schubert product of two characters and for some special skew characters the cover partition $\mC$, which is the partition which contains all partitions whose corresponding character appears in the product or the skew character.

\cite[Theorem 4.2]{Gut} states the following:

\begin{Sa}
\label{Sa:Mainskewschub}
Let $\mu,\l$ be partitions with $\mu\subseteq\l\subseteq (k^l)$ with some fixed integers $k,l$. Set $\l^{-1}=(k^l)/\l$.

Then: The coefficient of $[\a]$ in $[\l/\mu]$ equals the coefficient of $[\a^{-1}]=[(k^l)/\a]$ in $[\mu]\star_{(k^l)}[\l^{-1}]$
\end{Sa}

We will use that for $k\geq\mu_1+\nu_1,l\geq l(\mu)+l(\nu)$ the Schubert product $[\mu]\star_{(k^l)}[\nu]$ is the ordinary product $[\mu]\otimes[\nu]$.

Let us associate a skew diagram $\mA=\left((k^l)/\mu\right)^\circ)/\nu$ to partitions $\mu,\nu$. Here for the  Schubert product $k,l$ are fixed by $[\mu]\star_{(k^l)}[\nu]$ and for the ordinary product chosen as $k=\mu_1+\nu_1, l=l(\mu)+l(\nu)$. To obtain $\mA$ we remove from the rectangle $(k^l)$ the partition $\nu$ as usual and the partition $\mu$ rotated by $180^\circ$ from the lower right corner.

We know by definition that if  the box $(i,j)$ is in the base partition $\mB([\mA])$ then it is also in every partition $\a$ with $[\a]\in[\mA]$. So if we remove $\a$ from $(k^l)$ to get $\a^{-1}$ this box will be removed every time and so by Theorem~\ref{Sa:Mainskewschub} there cannot be a partition $\b$ (which would be a rotated $\a^{-1}$) with $[\b]\in[\mu]\star_{(k^l)}[\nu]$  containing the box $(k-i,l-j)$.

On the other hand if the box $(i,j)$ is not in the base partition then there is a partition $\a$ with $[\a]\in[\mA]$ without this box. So if we now remove $\a$ from $(k^l)$ to get $\a^{-1}$ this box is in $\a^{-1}$ and so there is a partition $\b$ (the rotated $\a^{-1}$) with $[\b]\in[\mu]\star_{(k^l)}[\nu]$ which does contain the box $(k-i,l-j)$.

So we get the following theorem:

\begin{Sa}
 \label{Sa:Cprod}
  Let $\mu,\nu$ be partitions and for fixed integers $k,l$ define the skew diagram $\mA=\left((k^l)/\mu\right)^\circ)/\nu$.

  Then: $\mC([\mu]\star_{(k^l)}[\nu])=\left((k^l)/\mB([\mA])\right)^\circ$
\end{Sa}

\begin{Bem}
 Since the Durfee size of the cover partition of the product is also the Durfee size of the product itself we can now easily read off the Durfee size of an arbitrary Schubert product. In~\cite{HL} we calculated the Durfee size only for some special kinds of Schubert products.
\end{Bem}

As noted above for the right choice of $k,l$ this gives us also the cover partition $\mC$ for the ordinary product.

We will give a short Example and calculate for $\m=(4,3,1)$ and $\nu=(5,2,2)$ the cover partition for the ordinary product $[\m]\otimes[\nu]$ and the Schubert product $[\mu]\star_{(7^4)}[\nu]$.

In the case of the ordinary product we first construct the skew diagram
\[\mA^1=(9^3,7^2,4)/(4,3,1)= \young(::::\h\h\h\h\h,:::\h\h\h\h\h\h,:\h\h\h\h\h\h\h\h,\h\h\h\h\h\h\h,\h\h\h\h\h\h\h,\h\h\h\h)\]
Using Remark~\ref{Bem:kurz} we get the base partition $\mB([\mA^1])=(8,7,6,4,3)$ and so for the cover partition
\[\mC([\m]\otimes[\nu])=((9^6)/\mB([\mA^1]))^\circ= \young(\h\h\h\h\h\h\h\h\h,\h\h\h\h\h\h\mB\mB\mB,\h\h\h\h\h\mB\mB\mB\mB,\h\h\h\mB\mB\mB\mB\mB\mB,\h\h\mB\mB\mB\mB\mB\mB\mB,\h\mB\mB\mB\mB\mB\mB\mB\mB)=(9,6,5,3,2,1)\]
where the boxes labeled $\mB$ form the partition $\mB([\mA^1])$.

Since we have $\mB([\m]\otimes[\nu])=(5,3,2)$ we now have that if $[\a]\in[\m]\otimes[\nu]$ then $\a$ has to satisfy $(5,3,2)\subseteq \a \subseteq (9,6,5,3,2,1)$. These upper and lower bounds are strict in the sense that there are no better bounds.

In the case of the Schubert product the skew diagram is:
\[\mA^2=(7,5,5,2)/(4,3,1)= \young(::::\h\h\h,:::\h\h,:\h\h\h\h,\h\h)\]
Using Remark~\ref{Bem:kurz} we get the base partition $\mB([\mA^2])=(4,2,1)$ and so for the cover partition
\[\mC([\m]\star_{(7^4)}[\nu])=((7^4)/\mB([\mA^2]))^\circ= \young(\h\h\h\h\h\h\h,\h\h\h\h\h\h\mB,\h\h\h\h\h\mB\mB,\h\h\h\mB\mB\mB\mB)=(7,6,5,3)\]
where the boxes labeled $\mB$ form the partition $\mB([\mA^2])$. From this we can also read off the Durfee size of the Schubert product as $d([\m]\star_{(7^4)}[\nu])=d(\mC([\mA^2]))=3$.

Since we have $(5,3,2)=\mB([\mu]\otimes[\nu])\subseteq \mB([\m]\star_{(7^4)}[\nu])$ we now have that if $[\a]\in[\m]\star_{(7^4)}[\nu]$ then $\a$ has to satisfy $(5,3,2) \subseteq \a \subseteq (7,6,5,3)$. Here the lower bound is not strict and explicit calculations show $\mB([\m]\star_{(7^4)}[\nu])=(5,4,2)$.

In the same way we can also construct the cover partition $\mC$ of skew characters if we restrict the skew diagram in the way that the associated Schubert product is in fact an ordinary product. If the skew diagram does not satisfy the constraints of the following theorem then we would only get a trivial upper bound for the cover partition.

\begin{Sa}
 \label{Sa:Cskew}
  Let $\mA=\l/\mu$ be a skew diagram with $\l=(\l_1^n,\l_{n+1},\ldots,\l_l), \mu_1\leq\l_l, l(\mu)\leq n $ and set $\l^{-1}=(\l_1^l)/\l$

 Then: $\mC([\mA])=\left(((\l_1)^l)/\mB([\mu]\otimes[\l^{-1}])\right)^\circ$
\end{Sa}

{\bfseries Acknowledgement:} John Stembridge's "SF-package for maple" \cite{stemmaple} was very helpful for computing examples.

\end{document}